\documentclass[a4paper]{article}
\usepackage{calligra}
\usepackage[T1]{fontenc}
\usepackage[english]{babel}
\usepackage{amsmath}
\usepackage{amsthm}
\usepackage{amssymb}
\usepackage{enumerate}
\usepackage{xspace}
\usepackage{euscript}
\usepackage{graphicx}
\usepackage{amscd}
\usepackage{epsfig}
\usepackage{tabularx}
\usepackage{pgfplots}
\usepackage{tikz}

\usepackage[notcite,notref,color]{showkeys}
\usepackage{xcolor}

\usetikzlibrary{datavisualization}
\usetikzlibrary{datavisualization.formats.functions}
\newtheorem{proposition}{Proposition}[section]
\newtheorem{definition}{Definition}[section]
\newtheorem{theorem}{Theorem}[section]

\newtheorem{rem}{Remark}[section]

\newcommand{\Supp}{{\rm Supp\,}}
\newcommand{\N}{{\mathbb N}}

\newcommand{\R}{{\mathbb R}}

\newcommand{\oline}{\overline}

\newcommand{\g}{\gamma}

\title{No loss of derivatives for hyperbolic operators with Zygmund-continuous coefficients in time}
\author{Ferruccio Colombini, Daniele Del Santo, Francesco Fanelli}
\date{\today}

\begin{document}

\maketitle
\section{Introduction}
Consider the second order strictly hyperbolic operator 
$$
L= \partial_t^2 -\sum_{j,k=1}^n \partial_j(a_{jk}(t,x)\partial_k ),
$$
with
$$
0<\lambda_0 |\xi|^2\leq \sum_{j,k=1}^n a_{jk}(t,x)\xi_j\xi_k \leq \Lambda_0|\xi|^2 \qquad \text{and}\qquad 
a_{jk}(t,x)=a_{kj}(t,x).
$$
It is well-known that, if the coefficients $a_{jk}$ are Lipschitz-continuous in $t$ and measurable in $x$, then the Cauchy problem related to $L$ is well-posed in the energy space.
In particular, a constant $C>0$ exists, such that
\begin{equation}\label{estnoloss}
\begin{array}{ll}
\displaystyle{\sup_{0\leq t \leq T} (\|u(t,\cdot)\|_{H^{1}}+ \|\partial_t u(t,\cdot)\|_{L^2})}\\[0.2 cm]
\qquad\qquad\displaystyle{
\leq C(\|u(0,\cdot)\|_{H^{1}}+ \|\partial_t u(0,\cdot)\|_{L^2}+\int_0^T \|Lu(s, \cdot)\|_{L^2}\, ds),}
\end{array}
\end{equation}
for all $u\in C([0,T]; H^1)\cap C^1([0,T]; L^2)$ with $Lu\in 
L^1([0,T]; L^2)$ (see \cite[Ch. IX]{H}, \cite{HS}).

In this note we are interested in second order strictly hyperbolic operators having \emph{non Lipschitz-continuous} coefficients
with respect to time.

After the pioneering paper by Colombini, De Giorgi and Spagnolo \cite{CDGS},
this topic 
has been widely studied.
A result of particular interest has been obtained in  \cite{CL}, where it was proved that, if the coefficients are $\log$-Lipschitz-continuous with respect to $t$ and $x$, i.e.
there exists $C>0$ such that
$$
\sup_{t,x} |a_{jk}(t+\tau, x+y)-a_{jk}(t,x)|\leq C (|\tau|+|y|)(1+\log \frac{1}{|\tau|+|y|}),
$$ 
then (\ref{estnoloss}) is no more valid, but the following weaker energy estimate can be recovered: 
\begin{equation}\label{estloss}
\begin{array}{ll}
\displaystyle{\sup_{0\leq t \leq T} (\|u(t,\cdot)\|_{H^{1-\theta-\beta t}}+ \|\partial_t u(t,\cdot)\|_{H^{-\theta-\beta t}})}\\[0.2 cm]
\qquad\displaystyle{
\leq C(\|u(0,\cdot)\|_{H^{1-\theta}}+ \|\partial_t u(0,\cdot)\|_{H^{-\theta}}+\int_0^T \|Lu(s, \cdot)\|_{H^{-\theta-\beta s}}\, ds),}
\end{array}
\end{equation}
for some constants $C>0$, $\beta>0$ and for all $u\in C^2([0,T]; H^\infty)$ and $\theta\in\; ]0,\,\frac{1}{4}[$.
Remark that, while in (\ref{estnoloss}) the norms of $u(t)$ and $\partial_t u(t)$ are estimated by the same norms of $u(0)$ and $\partial_t u(0)$, in (\ref{estloss}) the Sobolev spaces in which $u(t)$ and $\partial_t u(t)$ are measured are different and bigger than the spaces in which initial data are, so the estimate is less effective. This phenomenon goes under the name of {\it loss of derivatives}.
We refer e.g. to the introductions of \cite{CDSFM1}, \cite{CDSFM2} for more details and references about this problem.

Using a result obtained by Tarama in \cite{T} (see also Remark \ref{r:Tarama} below), it is possible to prove that  if the coefficients depend only on $t$ and are Zygmund-continuous, i.e.
\begin{equation} \label{hyp:Z}
 \sup_t |a_{jk}(t+\tau)+a_{jk}(t-\tau)-2a_{jk}(t)| \, dt \leq C_2 |\tau|,
\end{equation}
then (\ref{estnoloss}) is valid. Notice that the Zygmund assumption is weaker than the Lipschitz one. In \cite{CDSFM2}, the authors and M\'etivier proved that if the coefficients depend also on the space variable
and verify an isotropic Zygmund assumption (i.e. they are Zygmund-continuous both in time and space variables), then the Cauchy problem is well-posed with no loss, but only in the space
$H^{1/2}\times H^{-1/2}$. In particular, an estimate similar to \eqref{estnoloss} holds true, up to replacing the $H^1$ and $L^2$ norms respectively with the $H^{1/2}$ and $H^{-1/2}$ norms.
See also Remark \ref{r:iso-Zyg} below for more details.

The problem whether a Zygmund assumption both in time and space is still enough to recover well-posedness in general spaces $H^s\times H^{s-1}$ (and not only for $s=1/2$) remains at present
largely open.
As a partial step in this direction, in this note we consider a stronger hypothesis with respect to the space variable: namely we prove that, if the coefficients are Zygmund-continuous
with respect to $t$ and Lipschitz-continuous  with respect to $x$, then an estimate without loss of derivatives, similar to (\ref{estnoloss}), holds true. Then, the Cauchy problem relatd to
$L$ is well-posed in any space $H^s\times H^{s-1}$, for all $s\in\,]0,1]$.

Two are the main ingredients of the proof of our result. The first one is to resort to Tarama's idea of introducing a new type of energy associated to operator $L$: this new energy is
equivalent to the classical energy, but it contains a lower order term, whose goal is to produce special algebraic cancellations, which reveal to be fundamental in the energy estimates.
The second main ingredient, already introduced in \cite{CDSFM1} and \cite{CDSFM2}, is the use of paradifferential calculus with parameters (see e.g. \cite{MZ}, \cite{M2}), in order to deal
with coefficients depending also on $x$ and having low regularity in that variable.

\medbreak
We conclude this introduction with a short overview of the paper. In the next section we fix our hypotheses and state our main result, see Theorem \ref{theorem}. In Section \ref{s:tools} we collect some
elements of Littlewood-Paley theory, which are needed in the description of the functional classes where the coefficients belong to, and in the construction of paradifferential calculus with parameters.
With those tools at hand, we tackle the proof of Theorem \ref{theorem}, which is carried out in Section \ref{s:proof}.

\section{Main result} \label{s:result}
Given $T>0$ and an integer $n\geq 1$,  let $L$ be the linear differential operator defined on $[0,T]\times\R^n$ by
\begin{equation}
Lu= \partial_t^2 u -\sum_{j,k=1}^n \partial_j(a_{jk}(t,x)\partial_k u), \label{hyp eq}
\end{equation}
where, for all $j, \,k=1,\dots,n$,
\begin{equation}
a_{jk}(t,x)=a_{kj}(t,x), \label{sym cond}
\end{equation}
and there exist $\lambda_0,\,\Lambda_0>0$ such that
\begin{equation}
\lambda_0 |\xi|^2\leq \sum_{j,k=1}^n a_{jk}(t,x)\xi_j\xi_k \leq \Lambda_0|\xi|^2 ,\label{ellipticity cond}
\end{equation}
for all $(t,x)\in [0,T]\times \R^n$ and for all $\xi\in \R^n$. 
Suppose moreover that there exist constants $C_0,\, C_1>0$ such that, for all $j, \,k=1,\dots,n$ and for all $\tau\in\R$, $y\in \R^n$,
\begin{equation}
\sup_{t,x} |a_{jk}(t+\tau, x)+a_{jk}(t-\tau, x)-2a_{jk}(t,x)|\leq C_0|\tau|,
\label{Zyg cont}
\end{equation}
\begin{equation}
\sup_{t,x} |a_{jk}(t, x+y)-a_{jk}(t,x)|\leq C_1|y|.
\label{Lip cont}
\end{equation}

We can now state the main result of this paper.
\begin{theorem} \label{theorem}
Under the previous hypotheses, for all fixed $\theta\in [0,1[$, there exists a constant $C>0$, depending only on $\theta$ and $T$, such that
\begin{equation}
\begin{array}{ll}
\displaystyle{\sup_{0\leq t \leq T} (\|u(t,\cdot)\|_{H^{1-\theta}}+ \|\partial_t u(t,\cdot)\|_{H^{-\theta}})}\\[0.2 cm]
\qquad\qquad\displaystyle{
\leq C(\|u(0,\cdot)\|_{H^{1-\theta}}+ \|\partial_t u(0,\cdot)\|_{H^{-\theta}}+\int_0^T \|Lu(s, \cdot)\|_{H^{-\theta}}\, ds),}
\end{array}
\label{en est}
\end{equation}
for all $u\in C^2([0,T], H^{\infty}(\R^n))$.
\end{theorem}

Some remarks are in order.
\begin{rem} \label{r:Tarama}
If the coefficients $a_{jk}$ depend only on $t$, this result has been obtained by Tarama in \cite{T}, under the hypothesis that there exists a constant $C_2>0$ such that, for all $j, \,k=1,\dots,n$ and for all $\tau\in \,]0,T/2[$,
\begin{equation}
\int_\tau^{T-\tau} |a_{jk}(t+\tau)+a_{jk}(t-\tau)-2a_{jk}(t)| \, dt \leq C_2 \tau.
\label{hyp int}
\end{equation}
Tarama's hypothesis is weaker than ours, but, when coefficients depend also on the space variable, it is customary to take a pointwise condition with respect to time, like in (\ref{Zyg cont}) above
(see also \cite{CL}, \cite{CM}, \cite{CDSFM1}, \cite{CDSFM2} in this respect).
In particular, it is not clear at present whether or not the pointwise condition (\ref{Zyg cont}) can be relaxed to an integral one, similar to (\ref{hyp int}), in our framework.
\end{rem} 

\begin{rem} \label{r:iso-Zyg}
If the hypoteses (\ref{Zyg cont}) and  (\ref{Lip cont}) are replaced by the weaker following one:  there exists a constant $C_3>0$ such that, for all $j, \,k=1,\dots,n$ and for all $\tau\in\R$, $y\in \R^n$,
\begin{equation}
\sup_{t,x} |a_{jk}(t+\tau, x+y)+a_{jk}(t-\tau, x-y)-2a_{jk}(t,x)|\leq C_3(|\tau|+|y|),
\label{Zyg xt cont}
\end{equation}
the estimate (\ref{en est}) has been proved, only in the case of $\theta=1/2$, by the present authors and M\'etivier in \cite{CDSFM2}.
\end{rem} 

\begin{rem}
Assume (\ref{Zyg cont}) and the following hypothesis:  there exists a constant $C_4>0$ such that, for all $j, \,k=1,\dots,n$ and for all $y\in \R^n$ with $0<|y|\leq 1$,
\begin{equation}
\sup_{t,x} |a_{jk}(t, x+y)-a_{jk}(t,x)|\leq C_4|y| (1+\log \frac{1}{|y|}).
\label{log Lip cont}
\end{equation}
As a consequence of a result of the present authors and M\'etivier in \cite{CDSFM1} (stated for coefficients which are actually log-Zygmund with respect to time), one get that,
for all fixed $\theta\in \,]0,1[$, there exist a $\beta>0$, a time $T'>0$ and  a constant $C>0$ such that
\begin{equation}
\begin{array}{ll}
\displaystyle{\sup_{0\leq t \leq T'} (\|u(t,\cdot)\|_{H^{1-\theta-\beta t}}+ \|\partial_t u(t,\cdot)\|_{H^{-\theta-\beta t}})}\\[0.2 cm]
\qquad\displaystyle{
\leq C(\|u(0,\cdot)\|_{H^{1-\theta}}+ \|\partial_t u(0,\cdot)\|_{H^{-\theta}}+\int_0^{T'} \|Lu(s, \cdot)\|_{H^{-\theta-\beta s}}\, ds),}
\end{array}
\label{en with loss est}
\end{equation}
for all $u\in C^2([0,T'], H^{\infty}(\R^n))$. The condition (\ref{log Lip cont}) is weaker than (\ref{Lip cont}) but also 
(\ref{en with loss est}) is weaker than (\ref{en est}): (\ref{en with loss est}) has a loss of derivatives, while (\ref{en est}) performs no loss.
In addition, observe that (\ref{en with loss est}) holds only for $\theta\in \,]0,1[$, while (\ref{en est}) holds also for $\theta=0$. 
\end{rem} 
\section{Preliminary results} \label{s:tools}
We briefly list here some tools we will need in the proof of the main result. We follow closely the presentation of these topics given in  \cite{CDSFM1} and  \cite{CDSFM2}.
\subsection{Littlewood-Paley decomposition}
We will use the so called Littlewood-Paley theory. We refer to \cite{B}, \cite{Ch}, \cite{M} and \cite{BChD} for the details. 

We start recalling Bernstein's inequalities.
\begin{proposition}[{\cite[Lemma 2.2.1]{Ch}}] Let $0<r<R$. A constant $C$ exists so that, for all nonnegative integer $k$, all  $p,\,q\in[1,+\infty]$ with $p\leq q$ and for all function $u\in L^p(\R^d)$, we have, for all $\lambda>0$, 
\begin{itemize}
\item [i)] if $\Supp \hat u\subseteq B(0,\lambda R)=\{\xi\in \R^d\;:\, |\xi|\leq \lambda  R\}$, then
$$
\|\nabla^k u\|_{L^q}\leq C^{k+1}\lambda^{k+N(\frac{1}{p}-\frac{1}{q})}\| u\|_{L^p};
$$
\item [ii)] if $\Supp \hat u\subseteq C(0,\lambda r, \lambda R)= \{\xi\in \R^d\;:\, \lambda r\leq |\xi|\leq \lambda  R\}$, then
$$
C^{-k-1}\lambda^k\| u\|_{L^p}\leq \|\nabla^k u\|_{L^p}\leq C^{k+1}\lambda^k\| u\|_{L^p}.
$$
\end{itemize}
\end{proposition}
We introduce the dyadic decomposition. Let $\psi\in C^{\infty}([0,+\infty[, \R)$ such that $\psi$ is non-increasing and 
$$
\psi(t)=1\quad\text{for}\quad 0\leq t\leq\frac{11}{10}, \qquad \psi(t)=0\quad\text{for}\quad   t\geq\frac{19}{10}.
$$
We set, for $\xi\in\R^d$, 
\begin{equation}\label{defchipsi}
\chi(\xi)=\psi(|\xi|), \qquad \varphi(\xi)= \chi(\xi)-\chi(2\xi).
\end{equation}
We remark that the support of $\chi$ is contained in the ball $\{\xi\in \R^d\;:\, |\xi|\leq  2\}$, while that one of
$\varphi$ is contained in the annulus $\{\xi\in \R^d\;:\, 1/2 \leq |\xi|\leq   2\}$.

Given a tempered distribution $u$, the dyadic blocks are defined by
$$
\Delta_0 u= \chi(D)u ={\mathcal F}^{-1}(\chi(\xi)\hat u(\xi)), 
$$
$$ \Delta_j u= \varphi(2^{-j}D)u={\mathcal F}^{-1}(\varphi(2^{-j}\xi)\hat u(\xi))\quad \text{if} \quad j\geq 1,
$$
where we have denoted by ${\mathcal F}^{-1}$ the inverse of the Fourier transform. We introduce also the operator 
$$
S_ku=\sum_{j=0}^{k} \Delta_ju=  {\mathcal F}^{-1}(\chi(2^{-k}\xi)\hat u(\xi)).
$$
It is well known the characterization of classical Sobolev spaces via Littlewood-Paley decomposition: for any $s\in \R$, $u\in {\mathcal S}'$ is in $H^s$ if and only if, for all $j\in \N$, $\Delta_ju\in L^2$ and the series $\sum 2^{2js}\|\Delta_ju\|^2_{L^2}$ is convergent.
Moreover, in such a case, there exists a constant $C_s>1$ such that
\begin{equation}\label{charSob}
\frac{1}{C_s} \sum_{j=0}^{+\infty} 2^{2js}\|\Delta_ju \|_{L^2}^2\leq \|u\|^2_{H^s}\leq C_s \sum_{j=0}^{+\infty} 2^{2js}\|\Delta_ju \|_{L^2}^2.
\end{equation}

\subsection{Lipschitz, Zygmund and log-Lipschitz  functions}

In this subsection, we give a description of some functional classes relevant in the study of hyperbolic Cauchy problems. Namely, via Littlewood-Paley analysis,  we can characterise the spaces
of Lipschitz, Zygmund and log-Lipschitz functions. We start by recalling their definitions.

\begin{definition}
A function $u\in L^\infty(\R^d)$ is a Lipschitz-continuous function if
$$
|u|_{\rm Lip}=\sup_{x,y\in\R^d, \atop y\not =0}\frac{|u(x+y)-u(x)|}{|y|}<+\infty,
$$
$u$ is a Zygmund-continuous function if
$$
|u|_{\rm Zyg}=\sup_{x,y\in\R^d, \atop y\not =0}\frac{|u(x+y)+u(x-y)-2u(x)|}{|y|}<+\infty
$$
and, finally, $u$ is a log-Lipschitz-continuous function if
$$
|u|_{\rm LL}=\sup_{x,y\in\R^d, \atop 0<|y|\leq 1}\frac{|u(x+y)-u(x)|}{|y|(1+\log \frac{1}{|y|})}<+\infty.
$$
For $X\in\left\{{\rm Lip}\,,\;{\rm Zyg}\,,\;{\rm LL}\right\}$, we define $\|u\|_{\rm X}=\|u\|_{L^\infty}+|u|_{\rm X}$.
\end{definition}
\begin{proposition}
Let $u\in L^\infty(\R^d)$. We have the following characterisation:
\begin{align}
u\in {\rm Lip}(\R^d)\qquad&\text{if and only if}\qquad \sup_j\|\nabla S_ju\|_{L^\infty}<+\infty, \label{charLip} \\
u\in {\rm Zyg}(\R^d)\qquad&\text{if and only if}\qquad\sup_j 2^j\|\Delta_ju\|_{L^\infty}<+\infty, \label{charZyg}\\
u\in {\rm LL}(\R^d)\qquad&\text{if and only if}\qquad\sup_j\frac{\|\nabla S_ju\|_{L^\infty}}{j}<+\infty. \label{charlogLip}
\end{align}
\end{proposition}
\begin{proof}
The proof of (\ref{charZyg}) and (\ref{charlogLip}) can be found in \cite[Prop. 2.3.6]{Ch} and \cite[Prop. 3.3]{CL} respectively. We sketch the proof of (\ref{charLip}), for reader's convenience. Suppose $u\in {\rm Lip}(\R^d)$. We have
$$
\begin{array}{ll}
\displaystyle{D_{j}S_ku(x)}&=\displaystyle{ D_{j}({\mathcal F}^{-1}(\chi(2^{-k}\xi)\hat u(\xi)))(x)      }\\[0.2 cm]
&= \displaystyle{ {\mathcal F}^{-1}(\xi_j\chi(2^{-k}\xi)\hat u(\xi))(x)      }\\[0.2 cm]
&= \displaystyle{2^k {\mathcal F}^{-1}(2^{-k}\xi_j\chi(2^{-k}\xi)\hat u(\xi))(x)      }\\[0.2 cm]
&= \displaystyle{2^k \int_{\R^d}\theta_j(2^k y)  u(x-y)\, 2^{kd}dy     }
\end{array}
$$
where $\theta_j(y)= {\mathcal F}^{-1}(\xi_j\chi(\xi))(y)$. From the fact that $\int_{\R^d}\theta_j(y)\,dy=0$ we deduce that
$$
\begin{array}{ll}
\displaystyle{|D_{j}S_ku(x)|}&\leq \displaystyle{ 2^k|\int_{\R^d} \theta_j(2^k y)(  u(x-y) -u(x)) \, 2^{kd}dy | }\\[0.3 cm]
&\leq \displaystyle{|u|_{\rm Lip}\int_{\R^d}|\theta_j( y)||z|\, dz},
\end{array}
$$
hence $\sup_j\left\|\nabla S_ju\right\|_{L^\infty}<C\,|u|_{\rm Lip}$.

Conversely, let the second statement in (\ref{charLip}) hold. Remarking that
$$
D_{j}\Delta_ku (x)= {\mathcal F}^{-1}(\xi_j\varphi(2^{-k}\xi)\hat u(\xi))(x)= 
{\mathcal F}^{-1}(\xi_j(\chi(2^{-k}\xi)-\chi(2^{-k+1}\xi))\hat u(\xi))(x),
$$
and, by Bernstein's inequalities,
$$
|\Delta_k u (x)|\leq C 2^{-k+1}(\|\nabla S_ku\|_{L^\infty}+\|\nabla S_{k-1}u\|_{L^\infty}),
$$
we deduce that, for a new constant $C>0$,
$$
\|\Delta_ku\|_{L^\infty}\leq C2^{-k}\,\sup_j\left\|\nabla S_ju\right\|_{L^\infty}
$$
for all $k\geq 0$. Then
$$
\begin{array}{ll}
|u(x+y)-u(x)|&\displaystyle{\leq   |S_ku(x+y)-S_ku(x)|+|\sum_{h>k}(\Delta_hu(x+y)-\Delta_hu(x))|}\\[0.2 cm]
&\displaystyle{\leq \|\nabla S_ku\|_{L^\infty}|y|+2\sum_{h>k}\|\Delta_hu\|_{L^\infty}}\\[0.2 cm]
&\displaystyle{\leq C\,\sup_j\left\|\nabla S_ju\right\|_{L^\infty}\,(|y|+2^{-k}). }
\end{array}
$$
The conclusion follows from choosing $k$ in such a way that $2^{-k}\leq |y|$. 
\end{proof}

Notice that, going along the lines of the previous proof, we have actually shown that there exists $C_d>1$, depending only on $d$, such that, if $u\in {\rm Lip}(\R^d)$ then 
$$
\frac{1}{C_d}\, |u|_{\rm Lip} \leq \|\nabla S_ju\|_{L^\infty}\leq C_d \,|u|_{\rm Lip}. 
$$

\begin{proposition}[{\cite[Prop. 2.3.7]{Ch}}]
$$
{\rm Lip}(\R^d)\subseteq {\rm Zyg}(\R^d) \subseteq {\rm LL}(\R^d).
$$
\end{proposition}

In order to perform computations, we will need to smooth out our coefficients, because of their low regularity.
To this end, let us fix an even function $\rho\in C^\infty_0(\R)$ such that $0\leq \rho\leq 1$, $\Supp \rho\subseteq  [-1,\,1]$ and $\int_\R\rho(t)\, dt =1$, and define
$\rho_\varepsilon (t)= \frac{1}{\varepsilon}\rho( \frac{t}{\varepsilon})$.
The following result holds true.
\begin{proposition}[{\cite[Prop. 3.5]{CDSFM2}}]
Let $u\in {\rm Zyg}(\R)$. There exists $C>0$ such that,
\begin{align}
|u_{\varepsilon}(t)-u(t)|&\leq C|u|_{\rm Zyg}\,\varepsilon, \label{estconv1} \\[0.2 cm]
|u'_{\varepsilon}(t)|&\leq C|u|_{\rm Zyg}\,(1+\log\frac{1}{\varepsilon}), \label{estconv2} \\[0.2 cm]
|u''_{\varepsilon}(t)|&\leq C|u|_{\rm Zyg}\,\frac{1}{\varepsilon}, \label{estconv3}
\end{align}
where, for $0<\varepsilon\leq 1$, 
\begin{equation}\label{regintime}
u_\varepsilon (t)=(\rho_\varepsilon \ast u)(t)= \int_\R \rho_\varepsilon(t-s)u(s)\, ds.
\end{equation}
\end{proposition}
\subsection{Paradifferential calculus with parameters}
Let us sketch here the paradifferential calculus depending on a parameter $\gamma \geq 1$. The interested reader can refer to \cite[Appendix B]{MZ} (see also \cite{M2} and \cite {CM}).

Let $\gamma\geq 1$ and consider $\psi_\gamma\in C^\infty(\R^d\times \R^d)$ with the following properties
\begin{itemize}
\item[i)]
there exist $\varepsilon_1<\varepsilon_2<1$ such that
\begin{equation}\label{charpsi1}
\psi_\gamma(\eta,\xi)=\left\{
\begin{array}{ll}
1&\quad \text {for} \quad |\eta|\leq \varepsilon_1(\gamma+|\xi|),\\[0.2cm]
0&\quad \text {for} \quad |\eta|\geq \varepsilon_2(\gamma+|\xi|);
\end{array}
\right.
\end{equation}

\item[ii)] 
for all $(\beta, \alpha)\in \N^d\times \N^d$, there exists $C_{\beta, \alpha}\geq 0$ such that
\begin{equation}\label{charpsi2}
|\partial_\eta^\beta \partial_\xi^\alpha \psi_\gamma(\eta,\xi)|\leq C_{\beta,\alpha}(\gamma+|\xi|)^
{-|\alpha|-|\beta|} .
\end{equation}
\end{itemize}
The model for such a function will be 
\begin{equation}\label{psigamma}
\psi_\gamma(\eta,\xi)= \chi(\frac{\eta}{2^\mu})\chi(\frac{\xi}{2^{\mu+3}})+
\sum_{k=\mu+1}^{+\infty} \chi(\frac{\eta}{2^k})\varphi(\frac{\xi}{2^{k+3}}),
\end{equation}
where $\chi$ and $\varphi$ are defined in (\ref{defchipsi}) and $\mu$ is the integer part of  $\log_2 \gamma$. With this setting, we have that the constants $\varepsilon_1$, $\varepsilon_2$ and $C_{\beta, \alpha}$ in (\ref{charpsi1}) and (\ref {charpsi2}) does not depend on $\gamma$.

To fix ideas, from now on we take $\psi_\g$ as given in \eqref{psigamma}. Define now
$$
G^{\psi_\gamma}(x,\xi)=({\mathcal F}^{-1}_\eta \psi_\gamma )(x,\xi),
$$
where ${\mathcal F}^{-1}_\eta  \psi_\gamma$ is the inverse of the Fourier transform of $ \psi_\gamma$ with respect to the $\eta$ variable. 

\begin{proposition}[{\cite[Lemma 5.1.7]{M}}] 
For all $(\beta, \alpha)\in \N^d\times \N^d$, there exists $C_{\beta, \alpha}$, not depending on $\gamma$, such that
\begin{equation}\label{propG1}
\| \partial^\beta_x\partial^\alpha_\xi G^{\psi_\gamma}(\cdot, \xi)\|_{L^1(\R^d_x)}\leq  
C_{\beta,\alpha}(\gamma+|\xi|)^{-|\alpha|+|\beta|},
\end{equation}
\begin{equation}\label{propG2}
\|\,|\cdot| \,\partial^\beta_x\partial^\alpha_\xi G^{\psi_\gamma}(\cdot, \xi)\|_{L^1(\R^d_x)}\leq  
C_{\beta,\alpha}(\gamma+|\xi|)^{-|\alpha|+|\beta|-1}.
\end{equation}
\end{proposition}

Next, let $a\in L^\infty$. We associate to $a$ the classical pseudodifferential symbol
\begin{equation}\label{sigmaa}
\sigma_{a,\gamma}(x,\xi)= (\psi_\gamma(D_x,\xi)a)(x,\xi)= (G^{\psi_\gamma}(\cdot, \xi)\ast a)(x),
\end{equation}
and we define the paradifferential operator associate to $a$ as the classical pseudodifferential operator associated to $\sigma_{a,\gamma} $ (from now on, to avoid cumbersome notations, we will wright $\sigma_a$), i.e.
$$
T_a^\gamma u(x) = \sigma_a(D_x)u(x)= \frac{1}{(2\pi)^d}\int_{\R^d_\xi} \sigma_{a}(x,\xi)\hat u(\xi)\, d\xi.
$$
Remark that $T_a^1$ is the usual paraproduct operator
$$
T_a^1u = \sum_{k=0}^{+\infty} S_k a\Delta_{k+3}u,
$$
while, in the general case,
\begin{equation}\label{paraprod}
T_a^\gamma u = S_{\mu-1}aS_{\mu+2}u +\sum_{k=\mu}^{+\infty} S_k a\Delta_{k+3}u.
 \end{equation}
 with  $\mu$ equal to the integer part of  $\log_2 \gamma$.

In the following it will be useful to deal with Sobolev spaces on the parameter $\gamma\geq 1$.

\begin{definition} Let $\gamma\geq 1$ and $s\in \R$. 
We denote by $H^s_{\gamma}(\R^d)$ the set of tempered distributions $u$ such that
$$
\|u\|^2_{H^s_\gamma}=\int_{\R^d_\xi} (\gamma^2+|\xi|^2)^{s}|\hat u(\xi)|^2\,d\xi<+\infty.
$$
\end{definition} 
Let us remark that $H^s_\gamma=H^s$ and there exists $C_\gamma\geq1$ such that, for all $u\in H^s$,
$$
\frac{1}{C_\gamma} \|u\|^2_{H^s}\leq \|u\|^2_{H^s_\gamma}\leq C_\gamma \|u\|^2_{H^s}.
$$
\subsection{Low regularity symbols and calculus}
As in \cite{CDSFM1} and \cite{CDSFM2}, it is important to deal with paradifferential operators having symbols with limited regularity in time and space.  

\begin{definition}\label{defsymbol} A symbol of order $m$ is a function $a(t,x,\xi,\gamma)$ which is locally bounded on $[0,\,T]\times \R^n\times\R^n\times[1,+\infty[$, of class $C^\infty$ with respect to $\xi$ such that, for all $\alpha\in\N^n$, there exists $C_\alpha>0$ such that, for all $(t,x,\xi,\gamma)$,
\begin{equation}\label{estsym}
|\partial_\xi^\alpha a(t,x,\xi,\gamma)|\leq C_\alpha(\gamma+|\xi|)^{m-|\alpha|}.
\end{equation}
\end{definition} 

We take now a symbol $a$ of order $m\geq 0$, Zygmund-continuos with respect to $t$ uniformly with respect to $x$ and Lipschitz-continuos with respect to $x$ uniformly with respect to $t$.
We smooth out $a$ with respect to time as done in (\ref{regintime}), and call $a_\varepsilon$ the smoothed symbol.
We consider the classical symbol $\sigma_{a_\varepsilon}$ obtained from $a_\varepsilon$ via (\ref{sigmaa}).
In what follows, the variable $t$ has to be thought to as a parameter. 
\begin{proposition}\label{propsigma}
Under the previous hypotheses, one has:
$$
\begin{array}{lll}
|\partial_\xi^\alpha \sigma_{a_\varepsilon}(t,x,\xi,\gamma)|&\leq &C_\alpha(\gamma+|\xi|)^{m-|\alpha|},\\[0.4 cm]
|\partial_x^\beta\partial_\xi^\alpha \sigma_{a_\varepsilon}(t,x,\xi,\gamma)|&\leq &C_{\beta, \alpha}(\gamma+|\xi|)^{m-|\alpha|+|\beta|-1},\\[0.4 cm]
|\partial_\xi^\alpha \sigma_{\partial_t a_\varepsilon}(t,x,\xi,\gamma)|&\leq &\displaystyle{C_{\alpha}(\gamma+|\xi|)^{m-|\alpha|}\log(1+\frac{1}{\varepsilon}),}\\[0.4 cm]
|\partial_x^\beta\partial_\xi^\alpha \sigma_{\partial_t a_\varepsilon}(t,x,\xi,\gamma)|&\leq &\displaystyle{C_{\beta, \alpha}(\gamma+|\xi|)^{m-|\alpha|+|\beta|-1}\, \frac{1}{\varepsilon}},\\[0.4 cm]
|\partial_\xi^\alpha \sigma_{\partial^2_t a_\varepsilon}(t,x,\xi,\gamma)|&\leq &\displaystyle{C_{\alpha}(\gamma+|\xi|)^{m-|\alpha|}\,\frac{1}{\varepsilon},}\\[0.4 cm]
|\partial_x^\beta\partial_\xi^\alpha \sigma_{\partial^2_t a_\varepsilon}(t,x,\xi,\gamma)|&\leq &\displaystyle{C_{\beta, \alpha}(\gamma+|\xi|)^{m-|\alpha|+|\beta|-1}\, \frac{1}{\varepsilon^2}},
\end{array}
$$
where $|\beta|\geq 1$ and all the constants $C_{\alpha}$ and $C_{\beta, \alpha}$ don't depend on $\gamma$.
\end{proposition}
\begin{proof}
We have
$$
\sigma_{a_\varepsilon}(t,x,\xi,\gamma)=(G^{\psi_\gamma}(\cdot, \xi)\ast a_\varepsilon(t,\cdot,\xi,\gamma))(x),
$$
so that the first inequality follows from (\ref{propG1}) and (\ref{estsym}).

Next, we remark that
\begin{equation}\label{intG}
\int \partial_{x_j}G^{\psi_\gamma}(x,\xi)\, dx= \int{\mathcal F}^{-1}_\eta (\eta_j \psi_\gamma(\eta,\xi))(z)\, dz=
(\eta_j\psi(\eta, \xi))_{|\eta=0}=0.
\end{equation}
Consequently, using also (\ref{propG2}),
$$
\begin{array}{lll}
\displaystyle{| \partial_{x_j}\sigma_{a_\varepsilon}(t,x,\xi,\gamma)|}&=&\displaystyle{ |\int \partial_{y_j}G^{\psi_\gamma}(y,\xi)
 (a_\varepsilon(t, x-y, \xi, \gamma)-a_\varepsilon(t, x, \xi, \gamma))\, dy|},\\[0.4cm]
 &\leq & \displaystyle{   C\int | \partial_{y_j} G^{\psi_\gamma}(y,\xi)|\, |y|\,dy \; (\gamma+|\xi|)^m,    }\\[0.4cm]
&\leq & \displaystyle{ C (\gamma+|\xi|)^m}.
 \end{array}
$$
The other cases of the second inequality can be proved similarly.

The third inequality is again a consequence of (\ref{propG1}), keeping in mind  (\ref{estconv2}). It is in fact possible to prove that
$$
|\partial_\xi^\alpha \partial_t a_\varepsilon (t,x,\xi,\gamma)|\leq C_\alpha(1+\log\frac{1}{\varepsilon})(\gamma+|\xi|)^{m-|\alpha|} .
$$

Next, considering again 
(\ref{intG}),
we have
$$
\begin{array}{ll}
&\displaystyle{\partial_{x_j}\sigma_{\partial_t a_\varepsilon}(t,x,\xi,\gamma)}\\[0.3cm]
&=\displaystyle{ \int_{\R^n_y} \partial_{y_j}G^{\psi_\gamma}(y,\xi)
 (\partial_t a_\varepsilon(t, x-y, \xi, \gamma)-\partial_t a_\varepsilon(t, x, \xi, \gamma))\, dy},\\[0.5cm]
 &\leq  \displaystyle{   \int_{\R^n_y} \partial_{y_j}G^{\psi_\gamma}(y,\xi)\int_{\R_s}\frac{1}{\varepsilon^2}
 \rho ' (\frac{t-s}{\varepsilon})(a(s, x-y, \xi, \gamma)-  a(s, x, \xi, \gamma))\,ds\, dy
 }\\[0.5cm]
&\leq  \displaystyle{  \int_{\R_s}\frac{1}{\varepsilon^2}
 \rho ' (\frac{t-s}{\varepsilon}) \int_{\R^n_y} \partial_{y_j}G^{\psi_\gamma}(y,\xi)(a(s, x-y, \xi, \gamma)-  a(s, x, \xi, \gamma))\,dy\, ds
 }.
 \end{array}
$$
so that the fourth inequality easily follows.

The last two inequalities are obtained in similar way, using also  (\ref{estconv3}).
\end{proof}

To end this section it is worthy to recall some results on symbolic calculus. Again details can be found in \cite{CDSFM1}, \cite{CDSFM2} and \cite[Appendix B]{MZ}. 
\begin{proposition}[{\cite[Prop. 3.19]{CDSFM1}}]\label{propcalsym}
$\phantom.$
\begin{itemize}
\item[i)] Let $a$ be a symbol of order $m$ (see Def. \ref {defsymbol}). Suppose that $a$ is $L^\infty$ in the $x$ variable.
If we set  
$$
T_a u(x) =\sigma_a(D_x)u(x) = \frac{1}{(2\pi)^d}\int_{\R^d_\xi} \sigma_{a}(x,\xi, \gamma )\hat u(\xi)\, d\xi,
$$ 
then $T_a$ maps $H^{s}_\gamma$ into $H^{s-m}_\gamma$.
\item[ii)] Let $a$  and $b$ be two symbols of order $m$ and $m'$ respectively. Suppose that $a$ and $b$ are ${\rm Lip}$ in the $x$ variable. Then
$$
T_a\circ T_b=T_{ab}+R,
$$ 
and $R$ maps $H^{s}_\gamma$ into $H^{s-m-m'+1}_\gamma$.
\item[iii)] Let $a$ be a symbol of order $m$ which is ${\rm Lip}$ in the $x$ variable. Then, denoting by $T^*_a$ the $L^2$-adjoint operator of $T_a$, 
$$
T_a^\ast =T_{\oline a}+R,
$$
and $R$ maps $H^{s}_\gamma$ into $H^{s-m+1}_\gamma$.
\item[iv)] Let $a$ be a symbol of order $m$ which is ${\rm Lip}$ in the $x$ variable. Suppose 
$$
{\rm Re}\; a(x,\xi,\gamma)\geq \lambda_0 (\gamma+|\xi|)^m.
$$
with $\lambda_0>0$. Then there exists $\gamma_0\geq 1$, depending only on $\|a\|_{\rm Lip}$ and $\lambda_0$, such that, for all $\gamma\geq \gamma_0$ and for all $u\in H^\infty$, 
$$
{\rm Re}\;(T_au,u)_{L^2}\geq \frac{\lambda_0}{2}\|u\|^2_{H^{m/2}_\gamma}.
$$

\end{itemize}

\end{proposition}
\section{Proof of Theorem \ref{theorem}} \label{s:proof}
Also for the proof of the main result, we will closely follow the strategy implemented in \cite{CDSFM1} and \cite{CDSFM2}.

\subsection{Approximate energy} \label{ss:approx}
First of all we regularize the coefficients $a_{jk}$ with respect to $t$ via (\ref{regintime}) and we obtain $a_{jk, \varepsilon}$. We consider the 
$0$-th order symbol
$$
\alpha_\varepsilon(t,x,\xi,\gamma)= (\gamma^2+|\xi|^2)^{-\frac{1}{2}}(\gamma^2+\sum_{j,k} a_{jk,\varepsilon}(t,x)\xi_j\xi_k)^{\frac{1}{2}}.
$$
We fix
$$
\varepsilon=2^{-\nu},
$$
and we write $\alpha_\nu$ and $a_{jk,\,\nu}$ instead of $\alpha_{2^{-\nu}}$ and $a_{jk,\,2^{-\nu}}$ respectively.
From \hbox{Prop. \ref{propcalsym}}, point {\it iv)}, we have that there exists $\gamma\geq 1$ such that, for all $w\in H^\infty$,
$$
\|T^\gamma_{\alpha_{\nu}^{-1/2}} w\|_{L^2} \geq \frac{\lambda_0}{2}\|w\|_{L^2}\qquad\text{and}\qquad 
\|T^\gamma_{\alpha_{\nu}^{1/2}(\gamma^2+|\xi|^2)^{1/2}} w\|_{L^2}\geq \frac{\lambda_0}{2}\|w\|_{H^1_\gamma},
$$
where $\lambda_0$ has been defined in \eqref{ellipticity cond}.
We remark that $\gamma$ depends only on $\lambda_0$ and $\sup_{j,k}\|a_{jk}\|_{\rm Lip}$, in particular $\gamma$ does not depend on $\nu$. We fix such a $\gamma$ (this means also that $\mu$ is fixed in (\ref{paraprod})) and from now on we will omit to write it when denoting the operator $T$ and the Sobolev spaces $H^s$.

We consider $u\in C^2([0,T], H^\infty)$. We have
$$
\partial_t^2 u= \sum_{j,k}\partial_j(a_{jk}(t,x)\partial_k u)+Lu=
\sum_{j,k}\partial_j(T_{a_{jk}}\partial_k u)+\tilde Lu,
$$
where 
$$
\tilde Lu= Lu+\sum_{j,k}\partial_j((a_{jk}-T_{a_{jk}})\partial_k u).
$$
We apply the operator $\Delta_\nu$ and we obtain
$$
\partial_t^2 u_\nu= \sum_{j,k}\partial_j(T_{a_{jk}}\partial_k u_\nu)
+ \sum_{j,k}\partial_j([\Delta_\nu, T_{a_{jk}}]\partial_k u)+ (\tilde L u)_\nu,
$$
where $u_\nu=\Delta_\nu u$, $(\tilde L u)_\nu=\Delta_\nu(\tilde L u)$ and $\displaystyle{[\Delta_\nu, T_{a_{jk}}]}$ is the commutator between the localization operator $\Delta_\nu$ and the paramultiplication operator $T_{a_{jk}}$.

We set
$$
\begin{array}{ll}
&\displaystyle{v_\nu(t,x)= T_{\alpha^{-1/2}_{\nu} }\partial _t u_\nu- T_{\partial_t( \alpha^{-1/2}_{\nu})} u_\nu},\\[0.6 cm]
&\displaystyle{w_\nu(t,x)= T_{\alpha^{1/2}_{\nu}(\gamma^2+|\xi|^2)^{1/2 }}\ u_\nu},\\[0.5 cm]
&\displaystyle{z_\nu(t,x)= u_\nu},
\end{array}
$$
and we define the approximate energy associated to the $\nu$-th component  as
$$
e_\nu(t)=\|v_\nu(t, \cdot)\|^2_{L^2}+\|w_\nu(t, \cdot)\|^2_{L^2}+\|z_\nu(t, \cdot)\|^2_{L^2}.
$$
We fix $\theta\in [0,\,1[$ and define the total energy
$$
E_\theta(t)= \sum_{\nu=0}^{+\infty} \, 2^{-2\nu \theta} e_\nu(t).
$$
We remark that, as a consequence of Bernstein's inequalities,
$$
\|w_\nu\|^2_{L^2} \sim \|\nabla u_\nu\|^2_{L^2}\sim 2^{2\nu} \| u_\nu\|^2_{L^2}.
$$
Moreover, from (\ref{estconv2}) and, again, Bernstein's inequalities,
$$
 \|T_{\partial_t( \alpha^{-\frac{1}{2}}_{\nu})} u_\nu\|_{L^2}\leq C(\nu +1) \| u_\nu\|_{L^2} \leq C'\|w_\nu\|_{L^2},
$$
so that 
\begin{equation}\label{estdtunu}
\begin{array}{ll}
\|\partial_t u_\nu\|_{L^2}& \leq C  \|T_{ \alpha^{-1/2}_{\nu}} u_\nu\|_{L^2}\\[0.2 cm]
&\leq C(\|v_\nu\|_{L^2}+  \|T_{\partial_t( \alpha_{\nu}^{-1/2})} u_\nu\|_{L^2})\\[0.2 cm]
&\leq C(e_\nu(t))^{\frac{1}{2}}.
\end{array}
\end{equation}
We deduce that there exist constants $C_\theta$ and $C'_\theta$, depending only on $\theta$, such that
\begin{align*}
(E_\theta(0))^{\frac{1}{2}}&\leq C_\theta (\|\partial_tu(0)\|_{H^{-\theta}}+\|u(0)\|_{H^{1-\theta}}), \\
(E_\theta(t))^{\frac{1}{2}}&\geq C'_\theta (\|\partial_tu(t)\|_{H^{-\theta}}+\|u(t)\|_{H^{1-\theta}}).
\end{align*}

\subsection{Time derivative of the approximate energy}
We want to estimate the time derivative of $e_\nu$. 

Since
$$
\partial _t v_\nu = T_{\alpha_\nu^{-1/2}}\partial^2_t u_\nu - T_{\partial^2_t(\alpha_\nu^{-1/2})}u_\nu,
$$
we deduce
$$
\begin{array}{ll}
&\displaystyle{\frac{d}{dt} \|v_\nu(t)\|^2_{L^2}}\\[0.4cm]
&\displaystyle{=2\,{\rm Re} \big( v_\nu, T_{\alpha_\nu^{-1/2}}\partial^2_t u_\nu\big)_{L^2} -2\,{\rm Re} \big( v_\nu, T_{\partial^2_t(\alpha_\nu^{-1/2})}u_\nu\big)_{L^2}
}\\[0.4cm]
&\displaystyle{=-2\,{\rm Re} \big( v_\nu, T_{\partial^2_t(\alpha_\nu^{-1/2})}u_\nu\big)_{L^2}
+ 2\,{\rm Re} \big(v_\nu, \sum_{j,k}T_{\alpha_\nu^{-1/2}}\partial_j(T_{a_{jk}}\partial_k u_\nu)\big)_{L^2} 
}\\[0.4cm]
&\displaystyle{\phantom{=}+ 2\,{\rm Re} \big(v_\nu, \sum_{j,k}T_{\alpha_\nu^{-1/2}}\partial_j([\Delta_\nu,\,T_{a_{jk}}]\partial_k u)\big)_{L^2} + 2\,{\rm Re} \big(v_\nu, T_{\alpha_\nu^{-1/2}}(\tilde L u)_\nu\big)_{L^2}.
}
\end{array}
$$
We have
$$
\left|2\,{\rm Re} \big(v_\nu, T_{\alpha_\nu^{-1/2}}(\tilde L u)_\nu\big)_{L^2}\right|\leq C(e_\nu)^{\frac{1}{2}}\, \|(\tilde L u)_\nu\|_{L^2},
$$
and, from the fifth inequality in Prop. \ref{propsigma},  
$$
\left|2\,{\rm Re} \big(  v_\nu, T_{\partial^2_t(\alpha_\nu^{-1/2})}u_\nu\big)_{L^2}\right|\leq C\,
\|v_\nu\|_{L^2}\, 2^\nu \|u_\nu\|_{L^2}\leq C \, e_\nu(t).
$$
Therefore, we obtain
\begin{equation}\label{estdtvnu}
\begin{array}{ll}
\displaystyle{\frac{d}{dt} \|v_\nu(t)\|^2_{L^2}}
&\displaystyle{= 2\,{\rm Re} \big(v_\nu, \sum_{j,k}T_{\alpha_\nu^{-1/2}}\partial_j(T_{a_{jk}}\partial_k u_\nu)\big)_{L^2} 
}\\[0.4cm]
&\displaystyle{\phantom{=+}+ 2\,{\rm Re} \big(v_\nu, \sum_{j,k}T_{\alpha_\nu^{-1/2}}\partial_j([\Delta_\nu,\,T_{a_{jk}}]\partial_k u)\big)_{L^2} }\\[0.4cm]
&\displaystyle{\phantom{=+++}+2\,{\rm Re} \big(v_\nu, T_{\alpha_\nu^{-1/2}}(\tilde L u)_\nu\big)_{L^2}
+ Q_1,
}
\end{array}
\end{equation}
with $|Q_1|\leq C e_\nu(t)$.

\medskip
Next
$$
\partial _t w_\nu = T_{\partial_t(\alpha_\nu^{1/2})(\gamma^2+|\xi|^2)^{1/2}} u_\nu + T_{\alpha_\nu^{1/2}(\gamma^2+|\xi|^2)^{1/2}} \partial_t u_\nu,
$$
so that
$$
\begin{array}{ll}
&\displaystyle{\frac{d}{dt} \|w_\nu(t)\|^2_{L^2}}\\[0.4cm]
&\displaystyle{=2\,{\rm Re} \big(T_{\partial_t(\alpha_\nu^{1/2})(\gamma^2+|\xi|^2)^{1/2}} u_\nu, w_\nu\big)_{L^2}
+2\,{\rm Re} \big( T_{\alpha_\nu^{1/2}(\gamma^2+|\xi|^2)^{1/2}} \partial_t u_\nu, w_\nu\big)_{L^2}
}\\[0.4cm]
&\displaystyle{=2\,{\rm Re} \big(T_{\alpha_\nu(\gamma^2+|\xi|^2)^{1/2}} T_{-\partial_t(\alpha_\nu^{-1/2})}u_\nu, w_\nu\big)_{L^2}+ 2\,{\rm Re} \big( R_1 u_\nu, w_\nu\big)_{L^2}}\\[0.4cm]
&\displaystyle{\phantom{=}+2\,{\rm Re} \big( T_{\alpha_\nu(\gamma^2+|\xi|^2)^{1/2}} T_{\alpha_\nu^{-1/2}}\partial_t u_\nu, w_\nu\big)_{L^2} + 2\,{\rm Re} \big( R_2 u_\nu, w_\nu\big)_{L^2}
}\\[0.4cm]
&\displaystyle{=2\,{\rm Re} \big(v_\nu , T_{\alpha_\nu(\gamma^2+|\xi|^2)^{1/2}}w_\nu\big)_{L^2}
+ 2\,{\rm Re} \big(  v_\nu, R_3 w_\nu\big)_{L^2}
}\\[0.4cm]
&\displaystyle{\phantom{=}+ 2\,{\rm Re} \big( R_1 u_\nu, w_\nu\big)_{L^2}+ 2\,{\rm Re} \big( R_2 u_\nu, w_\nu\big)_{L^2}
}\\[0.4cm]
&\displaystyle{=2\,{\rm Re} \big(v_\nu , T_{\alpha_\nu^{-1/2}} T_{\alpha_\nu^{3/2}(\gamma^2+|\xi|^2)^{1/2}}w_\nu\big)_{L^2}
+ 2\,{\rm Re} \big( v_\nu, R_4 w_\nu\big)_{L^2}
}\\[0.4cm]
&\displaystyle{\phantom{=}+ 2\,{\rm Re} \big(  v_\nu, R_3 w_\nu\big)_{L^2}+2\,{\rm Re} \big( R_1 u_\nu, w_\nu\big)_{L^2}+ 2\,{\rm Re} \big( R_2 u_\nu, w_\nu\big)_{L^2}
}\\[0.4cm]
&\displaystyle{=2\,{\rm Re} \big(v_\nu , T_{\alpha_\nu^{-1/2}} T_{\alpha_\nu^2(\gamma^2+|\xi|^2)}u_\nu\big)_{L^2}
}\\[0.4cm]
&\displaystyle{\phantom{=}+ 2\,{\rm Re} \big( v_\nu , T_{\alpha_\nu^{-1/2}} R_5 u_\nu\big)_{L^2}+ 2\,{\rm Re} \big( v_\nu, R_4 w_\nu\big)_{L^2}
}\\[0.4cm]
&\displaystyle{\phantom{=}+ 2\,{\rm Re} \big(  v_\nu, R_3 w_\nu\big)_{L^2}+2\,{\rm Re} \big( R_1 u_\nu, w_\nu\big)_{L^2}+ 2\,{\rm Re} \big( R_2 u_\nu, w_\nu\big)_{L^2}.
}
\end{array}
$$
It is a straightforward computation, from the results of  symbolic calculus recalled in Prop. \ref{propcalsym}, to verify that all the operators $R_1$, $R_2$, $R_3$, $R_4$ and $R_5$ are $0$-th
order operators. Consequently,
\begin{equation}\label{estdtwnu}
\frac{d}{dt} \|w_\nu(t)\|^2_{L^2}=  2\,{\rm Re} \big(v_\nu , T_{\alpha_\nu^{-1/2}} T_{\alpha_\nu^2(\gamma^2+|\xi|^2)}u_\nu\big)_{L^2} + Q_2,
\end{equation}
with $|Q_2|\leq C e_\nu(t)$.

\medskip
Finally, from (\ref {estdtunu}),
\begin{equation}\label{estdtznu}
\frac{d}{dt} \|z_\nu(t)\|^2_{L^2}\leq |2\,{\rm Re} \big(u_\nu, \partial_t u_\nu\big)_{L^2}|\leq C e_\nu(t).
\end{equation}

Now we pair the first term in right hand side part of (\ref{estdtvnu}) with the first  term in right hand side part of  (\ref{estdtwnu}). We obtain
$$
\begin{array}{lll}
|2\,{\rm Re} \big(v_\nu, \sum_{j,k}T_{\alpha_\nu^{-1/2}}\partial_j(T_{a_{jk}}\partial_k u_\nu)\big)_{L^2}+
2\,{\rm Re} \big(v_\nu , T_{\alpha_\nu^{-1/2}} T_{\alpha_\nu^2(\gamma^2+|\xi|^2)}u_\nu\big)_{L^2}|\\[0.3cm]
\hfil \leq C\, \|v_\nu\|_{L^2}\,  \|\zeta_\nu\|_{L^2},
\end{array}
$$
where
$$
\begin{array}{lll}
\zeta_\nu=&\displaystyle{T_{\alpha_\nu^2(\gamma^2+|\xi|^2)}u_\nu +\sum_{j,k}\partial_j(T_{a_{jk}}\partial_k u_\nu)
}\\[0.3cm]
&\displaystyle{=T_{\gamma^2 +\sum_{j,k}a_{jk,\nu}\xi_j\xi_k }u_\nu +\sum_{j,k}\partial_j(T_{a_{jk}}\partial_k u_\nu)
}\\[0.3cm]
&\displaystyle{=T_{\gamma^2} u_\nu  +\sum_{j,k}( T_{a_{jk,\nu}\xi_j\xi_k }u_\nu + T_{\partial_j a_{jk}} \partial_k u_\nu - T_{a_{jk}\xi_j\xi_k } u_\nu).
}
\end{array}
$$
We have
$$
\|\sum_{j,k} T_{\partial_j a_{jk}} \partial_k u_\nu\|_{L^2} \leq C\sup_{j,k} \|a_{jk}\|_{\rm Lip} \| \nabla u_\nu\|_{L^2}
\leq C (e_\nu(t))^{\frac{1}{2}},
$$
and, from Bernstein's inequalities and (\ref{estconv1}),
$$
\|\sum_{j,k} T_{(a_{jk,\nu}-a_{jk})\xi_j\xi_k} u_\nu\|_{L^2} \leq C \sup_{j,k} \|a_{jk}\|_{\rm Lip}\, 2^{-\nu}\, \| \nabla^2 u_\nu\|_{L^2}\leq C (e_\nu(t))^{\frac{1}{2}}.
$$
From this we deduce
$$
 \|\zeta_\nu\|_{L^2}\leq C (e_\nu(t))^{\frac{1}{2}}.
$$

Summing up, from (\ref {estdtvnu}), (\ref {estdtwnu}) and (\ref {estdtunu}) we get
\begin{equation}\label{estdtenu}
\begin{array}{ll}
\displaystyle{\frac{d}{dt} \,e_\nu(t)}&\displaystyle{ \leq \;C_1 e_\nu(t)+ C_2 (e_\nu(t))^{\frac{1}{2}}\,\|(\tilde L u)_\nu\|_{L^2}}\\[0.3cm] 
&\qquad\qquad \qquad\displaystyle{+ |2\,{\rm Re} \big(v_\nu, \sum_{j,k}T_{\alpha_\nu^{-1/2}}\partial_j([\Delta_\nu,\,T_{a_{jk}}]\partial_k u)\big)_{L^2}|}.
\end{array}
\end{equation}
\subsection{Commutator estimate}
We want to estimate
$$
|\sum_{j,k} 2\,{\rm Re} \big(v_\nu,T_{\alpha_\nu^{-1/2}}\partial_j([\Delta_\nu,\,T_{a_{jk}}]\partial_k u)\big)_{L^2}|.
$$
We remark that
$$
\begin{array}{ll}
\displaystyle{[\Delta_\nu,\,T_{a_{jk}}]w}& \displaystyle{=\Delta_\nu(S_{\mu-1}a_{jk} \,S_{\mu+2}w) +  \Delta_\nu(\sum_{h=\mu}^{+\infty} S_h a_{jk}\, \Delta_{h+3}w)}\\[0.5cm]
&\phantom{=+} \displaystyle{-S_{\mu-1}a_{jk}\, S_{\mu+2}(\Delta_\nu w) -\sum_{h=\mu}^{+\infty} S_h a_{jk} \,
\Delta_{h+3}(\Delta_\nu w)
}\\[0.5cm]
& \displaystyle{=\Delta_\nu(S_{\mu-1}a_{jk} \,S_{\mu+2}w) -S_{\mu-1}a_{jk}\, \Delta_\nu (S_{\mu+2}w)}\\[0.3cm]
&\phantom{=+} \displaystyle{+\sum_{h=\mu}^{+\infty} \Delta_\nu(S_h a_{jk}\, \Delta_{h+3}w) -\sum_{h=\mu}^{+\infty} S_h a_{jk} \,
\Delta_\nu (\Delta_{h+3} w)
}\\[0.5cm]
& \displaystyle{= [\Delta_\nu,\,S_{\mu-1}a_{jk}]\,S_{\mu+2}w+\sum_{h=\mu}^{+\infty} [\Delta_\nu,\,S_h a_{jk}]
\, \Delta_{h+3} w},
\end{array}
$$
where we recall that $\mu$ is a fixed constant (depending on $\gamma$, which has been chosen at the beginning of Subsection \ref{ss:approx}). Hence we have
$$
\begin{array}{ll}
\partial_j([\Delta_\nu,&\!\!\!\!T_{a_{jk}}]\partial_k u)\\[0.2cm]
&\displaystyle{=\partial_j([\Delta_\nu,\,S_{\mu-1}a_{jk}]\,\partial_k(S_{\mu+2} u))+
\partial_j(\sum_{h=\mu}^{+\infty} [\Delta_\nu,\,S_h a_{jk}]
\,\partial_k( \Delta_{h+3} u)).}
\end{array}
$$

Consider first
$$
\partial_j([\Delta_\nu,\,S_{\mu-1}a_{jk}]\,\partial_k(S_{\mu+2} u)).
$$
The support of the Fourier transform of $[\Delta_\nu,\,S_{\mu-1}a_{jk}]\,\partial_k(S_{\mu+2} u)$ is contained in $\{|\xi|\leq 2^{\mu+4}\}$ and $[\Delta_\nu,\,S_{\mu-1}a_{jk}]\,\partial_k(S_{\mu+2} u)$ is identically $0$ if $\nu\geq \mu+5$. 
From Bernstein's inequalities and \cite[Th. 35]{CoM} we deduce that
$$
\| \partial_j([\Delta_\nu,\,S_{\mu-1}a_{jk}]\,\partial_k(S_{\mu+2} u))\|_{L^2}\leq C\, 2^\mu \,\sup_{j,k}\|a_{jk}\|_{\rm Lip}\, \|S_{\mu+2} u\|_{L^2}.
$$
We have
$$
\begin{array}{ll}
\displaystyle{\sum_{\nu=0}^{+\infty} 2^{-2\nu\theta} 
|\sum_{j,k} 2\,{\rm Re} \big(v_\nu,T_{\alpha_\nu^{-1/2}}\partial_j([\Delta_\nu,\,S_{\mu-1}a_{jk}]\,\partial_k(S_{\mu+2} u))\big)_{L^2}
|}\\[0.3cm]
\qquad\qquad\qquad\quad\leq \displaystyle{C \, 2^\mu \,\sup_{j,k}\|a_{jk}\|_{\rm Lip}\, \sum_{\nu=0}^{\mu+4} 
2^{-2\nu\theta}\|v_\nu\|_{L^2} ( \sum_{h=0}^{\mu+2} \|u_h\|_{L^2})}\\[0.3cm]
\qquad\qquad\qquad\quad\leq \displaystyle{C \, 2^{\mu +(\mu+4)\theta}\, \sup_{j,k}\|a_{jk}\|_{\rm Lip}\, \sum_{\nu=0}^{\mu+4} 
2^{-\nu\theta}\|v_\nu\|_{L^2} \sum_{h=0}^{\mu+4} 2^{-h\theta}\|u_h\|_{L^2}}\\[0.3cm]
\qquad\qquad\qquad\quad\leq \displaystyle{C
\,\sup_{j,k}\|a_{jk}\|_{\rm Lip}\, 
\sum_{h=0}^{\mu+4} 2^{-2\nu\theta}\, e_\nu(t).
}
\end{array}
$$

Consider then
$$
\partial_j(\sum_{h=\mu}^{+\infty}\, [\Delta_\nu,\,S_h a_{jk}]
\,\partial_k( \Delta_{h+3} u)).
$$
Looking at the support of the Fourier transform, it is possible to see that $$[\Delta_\nu,\,S_h a_{jk}]
\,\partial_k( \Delta_{h+3} u)$$ is identically $0$ if $|h+3-\nu|\geq 3$. As a consequence, the sum over $h$ is reduced to at most $5$ terms: $\partial_j([\Delta_\nu,\,S_{\nu-5} a_{jk}] \,\partial_k( \Delta_{\nu-2} u))$, \dots,
$\partial_j([\Delta_\nu,\,S_{\nu-1} a_{jk}] \,\partial_k( \Delta_{\nu+2} u))$. Each of these terms has the support of the Fourier transform contained in the ball $\{ |\xi|\leq 2^{\nu+4}\}$. 

We consider the term $\partial_j([\Delta_\nu,\,S_{\nu-3} a_{jk}] \,\partial_k( \Delta_{\nu} u))$: for the other terms the estimate will be similar. Again by  Bernstein's inequalities and \cite[Th. 35]{CoM} we infer 
$$
\| \partial_j([\Delta_\nu,\,S_{\nu-3}a_{jk}]\,\partial_k(\Delta_\nu u))\|_{L^2}\leq C\, 2^\nu \,\sup_{j,k}\|a_{jk}\|_{\rm Lip}\, \|\Delta_{\nu} u\|_{L^2},
$$
and then
$$
\begin{array}{ll}
\displaystyle{| 
\sum_{j,k} 2\,{\rm Re} \big(v_\nu,T_{\alpha_\nu^{-1/2}}\partial_j(\sum_{h=\mu}^{+\infty}\, [\Delta_\nu,\,S_h a_{jk}]
\,\partial_k( \Delta_{h+3} u))\big)_{L^2}
|}\\[0.3cm]
\quad\leq \displaystyle{C \,\sup_{j,k}\|a_{jk}\|_{\rm Lip}\, 
(e_{\nu-2}(t)+e_{\nu-1}(t)+e_{\nu}(t)+e_{\nu+1}(t)+e_{\nu+2}(t)).
}
\end{array}
$$
Thus we have
$$
\begin{array}{ll}
\displaystyle{\sum_{\nu=0}^{+\infty} 2^{-2\nu\theta} 
|\sum_{j,k} 2\,{\rm Re} \big(v_\nu,T_{\alpha_\nu^{-1/2}}\partial_j(\sum_{h=\mu}^{+\infty}\, [\Delta_\nu,\,S_h a_{jk}]
\,\partial_k( \Delta_{h+3} u))\big)_{L^2}
|}\\[0.3cm]
\qquad\qquad\qquad\qquad\qquad\qquad\qquad\qquad\leq \displaystyle{C  \,\sup_{j,k}\|a_{jk}\|_{\rm Lip}\, 
\sum_{\nu=0}^{+\infty} 2^{-2\nu\theta} e_\nu(t).
}
\end{array}
$$

As a conclusion
\begin{equation}\label{estcomm}
\sum_{\nu=0}^{+\infty} 2^{-2\nu\theta} 
|\sum_{j,k} 2\,{\rm Re} \big(v_\nu,T_{\alpha_\nu^{-1/2}}\partial_j([\Delta_\nu,\,T_{a_{jk}}]\partial_k u)\big)_{L^2}|
\leq C_3 \sum_{\nu=0}^{+\infty} 2^{-2\nu\theta} e_\nu(t),
\end{equation}
where $C_3$ depends on $\gamma$, $\theta$ and $\sup_{j,k}\|a_{jk}\|_{\rm Lip}$.
\subsection{Final estimate}
From (\ref{estdtenu}) and (\ref{estcomm}) we obtain
$$
\begin{array}{ll}
\displaystyle{\frac{d}{dt} E_\theta(t)}&\displaystyle{\leq (C_1+C_3) \sum_{\nu=0}^{+\infty} \, 2^{-2\nu\theta} e_\nu(t)
+C_2 \sum_{\nu=0}^{+\infty} \, 2^{-2\nu\theta} (e_\nu(t))^{\frac{1}{2}}\| (\tilde Lu(t))_\nu\|_{L^2}}\\[0.5cm]
&\displaystyle{\leq (C_1+C_3) \sum_{\nu=0}^{+\infty} \, 2^{-2\nu\theta} e_\nu(t)
+C_2 \sum_{\nu=0}^{+\infty} \, 2^{-2\nu\theta} (e_\nu(t))^{\frac{1}{2}}\| (Lu(t))_\nu\|_{L^2}}\\[0.5cm]
&\displaystyle{\qquad\qquad\
+C_2 \sum_{\nu=0}^{+\infty} \, 2^{-2\nu\theta} (e_\nu(t))^{\frac{1}{2}}\| \big(
\sum_{j,k} \partial_j( (a_{jk}-T_{a_{jk}})\partial_ku)\big)_\nu\|_{L^2}}.
\end{array}
$$
We have
$$
\begin{array}{ll}
\displaystyle{\sum_{\nu=0}^{+\infty} \, 2^{-2\nu\theta} (e_\nu(t))^{\frac{1}{2}}\| \big(
\sum_{j,k} \partial_j( (a_{jk}-T_{a_{jk}})\partial_ku)\big)_\nu\|_{L^2}}\\[0.5cm]
\displaystyle{\quad \leq \Big(\sum_{\nu=0}^{+\infty} \, 2^{-2\nu\theta} e_\nu(t) \Big)^{\frac{1}{2}}
\Big(\sum_{\nu=0}^{+\infty} \, 2^{-2\nu\theta}\| \big(
\sum_{j,k} \partial_j( (a_{jk}-T_{a_{jk}})\partial_ku)\big)_\nu\|^2_{L^2} \Big)^{\frac{1}{2}}}.
\end{array}
$$
From (\ref{charSob}) we deduce
$$
\sum_{\nu=0}^{+\infty} \, 2^{-2\nu\theta}\| \big(
\sum_{j,k} \partial_j( (a_{jk}-T_{a_{jk}})\partial_ku)\big)_\nu\|^2_{L^2} \leq C\, 
\| 
\sum_{j,k} \partial_j( (a_{jk}-T_{a_{jk}})\partial_ku)\|^2_{H^{-\theta}}
$$
Now, using \cite[Prop. 3.5]{DSP} in the case $\theta\in\;]0,\,1[$ and \cite[Th. 5.2.8]{M}  in the case $\theta =0$,
$$
\| \sum_{j,k} \partial_j( (a_{jk}-T_{a_{jk}})\partial_ku)\|^2_{H^{-\theta}}\leq C (\sup_{j,k}\|a_{jk}\|_{\rm Lip}) \|u(t)\|_{H^{1-\theta}},
$$
so that
$$
\sum_{\nu=0}^{+\infty} \, 2^{-2\nu\theta} (e_\nu(t))^{\frac{1}{2}}\| \big(
\sum_{j,k} \partial_j( (a_{jk}-T_{a_{jk}})\partial_ku)\big)_\nu\|_{L^2}\leq C_4 E_\theta(t),
$$
and finally
$$
\frac{d}{dt} E_\theta(t)\leq C ( E_\theta(t) + (E_\theta(t))^\frac{1}{2}\|Lu(t)\|_{H^{-\theta}}).
$$
The energy estimate (\ref{en est}) easily follows from this last inequality and Gr\"onwall Lemma.

\end{document}